\documentclass[a4paper,10pt]{article}

\usepackage{hyperref}
\usepackage[scaled=0.90]{helvet}
\hypersetup{colorlinks=false}

\usepackage{amscd,mathrsfs}
\input xy
\xyoption{all}

\usepackage{amsmath}
\usepackage{amsthm}
\usepackage{amssymb}

\newtheorem{thm}{Theorem}[section]
\newtheorem{cor}[thm]{Corollary}
\newtheorem{lem}[thm]{Lemma}
\newtheorem{prop}[thm]{Proposition}

\theoremstyle{definition}
\newtheorem{rem}[thm]{Remark}

\newtheorem{defi}[thm]{Definition}

\newcommand{\supp}[1]{\operatorname{supp} #1}
\newcommand{\Ann}[1]{\operatorname{Ann}(#1)}

\newcommand{\K}{\mathbb{K}}
\newcommand{\R}{\mathbb{R}}
\newcommand{\Q}{\mathbb{Q}}
\newcommand{\N}{\mathbb{N}}

\newcommand{\C}{\mathbb{C}}
\newcommand{\Kt}{\widetilde{\K}}

\newcommand{\Rt}{\widetilde{\R}}

\newcommand{\Cinf}{\mathcal{C}^{\infty}}

\newcommand{\A}{\mathcal{A}}

\newcommand{\EM}{\mathcal{E}_M}
\newcommand{\Neg}{\mathcal{N}}
\newcommand{\eps}{\varepsilon}

\numberwithin{equation}{section} 

\newcommand{\Gco}{\mathcal{G}_{co}}
\newcommand{\Gsm}{\mathcal{G}_{sm}}

\newcommand{\comp}{\subset\subset}
\newcommand{\cinfty}{{\mathcal C}^\infty}
\newcommand{\Emco}{\mathcal{E}_{M,co}}
\newcommand{\Emsm}{\mathcal{E}_{M,sm}}
\newcommand{\Nco}{\mathcal{N}_{co}}
\newcommand{\Nsm}{\mathcal{N}_{sm}}

\newcommand{\G}{\mathcal{G}}
\newcommand{\Kco}{\widetilde{\K}_{co}}
\newcommand{\Ksm}{\widetilde{\K}_{sm}}
\newcommand{\Em}{\mathcal{E}_M}
\newcommand{\Rsm}{\widetilde{\R}_{sm}}
\newcommand{\Csm}{\widetilde{\C}_{sm}}
\newcommand{\Rco}{\widetilde{\R}_{co}}
\newcommand{\Cco}{\widetilde{\C}_{co}}

\newcommand{\co}{\mathcal{C}}

\title{Algebras of generalized functions with smooth parameter dependence}
\author{Annegret Burtscher \footnote{Electronic mail: annegret.burtscher@univie.ac.at, 
supported by research stipend FS 506/2010 of the University of Vienna
}\\ 
        Michael Kunzinger \footnote{Electronic mail: michael.kunzinger@univie.ac.at, supported by START-project Y237
        and FWF-project P20525 of the Austrian Science Fund}\\
         {\small Department of Mathematics, University of Vienna}\\
         {\small Nordbergstr. 15, A-1090 Wien, Austria}\\
}

\begin{document}
\date{}
\maketitle

\begin{abstract}

We show that spaces of Colombeau generalized functions with smooth
parameter dependence are isomorphic to those with continuous
parametri\-za\-tion. Based on this result we initiate a systematic study of
algebraic properties of the ring $\Ksm$ of generalized numbers in this
unified setting. In particular, we investigate the ring and order
structure of $\Ksm$ and establish some properties of its ideals.
  
\end{abstract}

\section{Introduction}

Algebras of generalized functions, in particular Colombeau algebras are a
versatile tool for studying singular problems in analysis, geometry and
mathematical physics (cf., e.g., \cite{C84,C85,NPS98,GKOS2001}). 
Over the past decade, there has
been increased interest in the structural theory of such algebras, in
particular concerning topological and functional analytic aspects of the theory
(e.g., \cite{S0,S,G1,G2,GV}).
Furthermore, starting with the fundamental paper \cite{AJ01}, algebraic
properties, both of the ring of Colombeau generalized functions and of
Colombeau algebras have become a main line of research (\cite{AJ01,AJOS08,V09,V10}). 

From the very outset, certain questions of an algebraic nature have played
an important role in Colombeau theory. Among them is the solution of
algebraic equations in generalized functions. In the standard (special or full) 
version of the theory, polynomials have additional roots when
considered as generalized functions. These roots are obtained by mixing
classical roots. For example, apart from its classical solutions $\pm 1$,
the equation $x^2=1$ additionally has the generalized root given by the
equivalence class of $(x_\eps)_\eps$ with $x_\eps = 1$ for $\eps\in \Q$
and $x_\eps=-1$ for $\eps\not\in \Q$. Usually, such additional roots are an
unwanted phenomenon (cf.\ the discussion in \cite{B90}, Ch.\ 1.10). They can be
avoided by demanding continuous dependence of representatives on the
regularization parameter $\eps$ (see \cite[Prop.\ 12.2]{O92}). More
generally, one can show that algebraic equations only possess classical
solutions in a setting with continuous parameter dependence (\cite{MP}).

Apart from avoiding pathological solutions of algebraic equations, there
are a number of intrinsic reasons for studying Colombeau spaces with
continuous or smooth parameter dependence. To begin with, when considering
full versions of the construction smooth in the test function variable, 
as done e.g.,  in \cite{C84, GKOS2001}, smooth dependence on all variables is automatic. 
This is inherited by special Colombeau algebras when these are considered as
subspaces of such full algebras (cf.\ \cite{O92}, p.\ 111).
Smooth dependence on the regularization parameter is, in fact, built in in 
the image of the space of distributions within the Colombeau algebra. 
Indeed, regularization via convolution yields as the embedded image of a 
(say, compactly supported) distribution $w$ the net $(w*\rho_\eps)_\eps$,
where $\rho_\eps = 1/\eps^n \rho(./\eps)$ and $\rho$ is an $\mathcal S$-mollifier
with all higher moments vanishing, 
which is obviously smooth in $\eps$. Thus it is natural to require 
the same regularity for all elements of the Colombeau algebra (or its ring of constants, respectively).

Moreover, certain geometrical constructions in special Colombeau algebras
require smooth parameter dependence. We mention, in particular, the notion
of generalized vector fields along a generalized curve (which is needed to
model geodesics in singular space-times in general relativity), cf.\ \cite{KS2002,KSV2003},
or sheaf properties in spaces of manifold-valued generalized functions
(cf.\ \cite{KSV2009}).

Finally, we point out the important characterization result on
isomorphisms of Colombeau algebras on differentiable manifolds due to H.\
Vernaeve. He proved that, up to multiplication by an idempotent
generalized number, multiplicative linear functionals on a Colombeau
algebra are precisely given as evaluation maps in generalized points (see
\cite[Th.\ 4.5]{V09}) and algebra isomorphisms are realized as pullbacks under invertible
manifold-valued generalized function (\cite[Th.\ 5.1]{V09}). When transferring these results to
the case of smooth parameter dependence, due to the fact that there are no
nontrivial idempotents in this setting (cf.\ Proposition \ref{gen:idem2} below), both characterizations hold 
without restriction (\cite{B2009,B2010}).

The purpose of the present paper is to initiate a systematic study of
special Colombeau algebras with continuous or smooth parameter dependence.
It is structured as follows: After fixing some notation in Section 2,
the main result of the first part of the paper is given in Section 3,
namely that Colombeau spaces
with continuous or smooth parameter dependence are in fact isomorphic.
Based on this identity, Section 4 studies algebraic
properties of the space $\Ksm$ of smoothly parametrized generalized
numbers. In particular, we analyze the ring structure of $\Ksm$ (zero divisors,
exchange ring, Gelfand ring, and partial order) and conclude by establishing some
fundamental properties of ideals in $\Ksm$.

\section{Notation}\label{notation}

 Throughout this paper we will write $I$ for the interval $(0,1]$. The manifolds $M$ and $N$ are assumed to be smooth, Hausdorff and second countable. For any two sets $A$ and $B$ the relation $A \comp B$ denotes that $A \subseteq \overline A \subseteq B$ with $\overline A$ compact. Whenever we do not have to distinguish between $\R$ and $\C$ we will denote either of the fields by $\K$.
 
The prototypical special Colombeau algebra of generalized functions over some smooth manifold $M$ is given as the quotient $\G(M):=\Em(M)/\Neg(M)$, where the algebra $\Em(M)$ and the ideal $\Neg(M)$ of $\Em(M)$ are defined by 
(with ${\mathcal P}(M)$ the space of linear differential operators on $M$)
\begin{eqnarray*}
\Em(M) &:=& \{(u_\eps)_\eps\in {\mathcal C}^\infty(M)^I \mid \forall K\subset\subset X\ \forall P\in{\mathcal P}(M)\ \exists N\in\N: \\
&& \hphantom{xxxxxxxxxxxxxxx} \sup_{x\in K}|Pu_\eps(x)|=O(\eps^{-N})\} \\
\Neg(M) &:=& \{(u_\eps)_\eps\in {\mathcal C}^\infty(M)^I \mid \forall K\subset\subset X\ \forall P\in{\mathcal P}(M)\ \forall m\in\N: \\
&& \hphantom{xxxxxxxxxxxxxxx} \sup_{x\in K}|Pu_\eps(x)|=O(\eps^{m})\} 
\end{eqnarray*}
The corresponding ring of constants in $\G(M)$ is given as $\widetilde\K := \Em/\Neg$, where
\begin{eqnarray*}
\Em &=& \{(r_\eps)_\eps\in \K^I \mid \exists N\in\N: |r_\eps|=O(\eps^{-N})\} \\
\Neg &=& \{(r_\eps)_\eps\in \K^I \mid \forall m\in\N: |r_\eps|=O(\eps^m)\}
\end{eqnarray*}
The equivalence class of some representative $(u_\eps)_\eps$ is denoted by $[(u_\eps)_\eps]$.
In the above definitions, the representatives $(u_\eps)_\eps$ and $(r_\eps)_\eps$ are allowed to depend
arbitrarily on the regularization parameter $\eps$. If instead we consider representatives that
depend continuously or smoothly on $\eps$ (i.e., $(\eps,x) \mapsto u_\eps(x)$ is continuous in $\eps$ and
smooth in $x$, or smooth in both variables, respectively, and analogously for $\eps \mapsto r_\eps$)
we denote this by the following subscripts: none (any parametrization, which is the standard definition), $_{co}$ (continuous parametrization), $_{sm}$ (smooth parametrization). Moderateness and negligibility are denoted by $\Em$, $\Emco$, $\Emsm$ and $\Neg$, $\Nco$, $\Nsm$, respectively. The rings of \emph{generalized numbers} are $\Kt$, $\Kco$ and $\Ksm$. Given two manifolds $M$ and $N$ we write $\G(M)$, $\Gco(M)$ and $\Gsm(M)$ for the special Colombeau algebras and $\G[M,N]$, $\Gco[M,N]$ and $\Gsm[M,N]$ for the spaces of manifold-valued generalized functions. We refer to \cite{GKOS2001,K2002, KSV2003} for details on these spaces.

By $\tau_{co}$ and $\tau_{sm}$ we denote the natural homomorphisms between spaces of generalized numbers and functions with continuous, smooth or arbitrary dependence on $\eps$. For simplicity, we do not distinguish notationally between these homomorphisms on different domains: $\tau_{co}$ will always denote maps from spaces with continuous to such with general parametrization, and $\tau_{sm}$ maps from spaces with smooth to the corresponding ones with continuous parametrization, e.g., $\tau_{co}: \Kco \rightarrow \Kt$ and $\tau_{sm}: \Ksm \rightarrow \Kco$, etc. We will sometimes use $\tau$ if a distinction is not necessary.

\section{Smooth, continuous, and arbitrary pa\-ra\-me\-tri\-zation}

 In this section we examine the interrelation between the various versions of spaces of generalized functions and generalized numbers introduced in Section~\ref{notation}. In particular, we shall prove that $\Ksm \cong \Kco \subsetneq \Kt$ and $\Gsm(M) \cong \Gco(M) \subsetneq \G(M)$.


 To begin with we note that $\Ksm \subseteq \Kco \subseteq \Kt$ via the canonical embeddings $\tau_{co}$ and $\tau_{sm}$, defined by $[(r_{\eps})_{\eps}] \mapsto [(r_{\eps})_{\eps}]$:
$$\xymatrix{\Ksm \, \ar@/_1cm/@{^(->}[rrrr]^{\tau_{co} \circ \tau_{sm}} \ar@{^(->}[rr]^{\tau_{sm}} \, && \Kco \ar@{^(->}[rr]^{\tau_{co}} \, && \Kt},$$
 These maps are well-defined as $\Emsm \subseteq \Emco \subseteq \EM$ and $\Nsm \subseteq \Nco \subseteq \Neg$: if $(s_{\eps})_{\eps}$ is another representative of $r$ then $\tau([(s_{\eps})_{\eps}]) = \tau([(r_{\eps})_{\eps}])$. Moreover, $\tau_{co}$, $\tau_{sm}$ and therefore also $\tau_{co} \circ \tau_{sm}$ are ring homomorphisms. They are injective because $\Emco \cap \Neg \subseteq \Nco$ and $\Emsm \cap \Nco \subseteq \Nsm$. Thus we obtain:

\begin{lem} \label{emb}
 The maps $\tau_{co}: \Kco \rightarrow \Kt$, $\tau_{sm}: \Ksm \rightarrow \Kco$ and $\tau_{co} \circ \tau_{sm}: \Ksm \rightarrow \Kt$, defined by $[(r_{\eps})_{\eps}] \mapsto [(r_{\eps})_{\eps}]$, are injective and unital ring homomorphisms. 
\end{lem}

 Let $M$ be a smooth, Hausdorff and second countable manifold. As for generalized numbers we consider the following maps between the different versions of algebras of generalized functions:
$$\xymatrix{\Gsm(M) \;\; \ar@/_1cm/@{^(->}[rrrr]^{\tau_{co} \circ \tau_{sm}} \ar@{^(->}[rr]^{\tau_{sm}} \, && \Gco(M) \ar@{^(->}[rr]^{\tau_{co}} \, && \G(M)}.$$

As above we obtain:

\begin{lem} \label{gf:emb1}
 Let $M$ be a manifold. The maps $\tau_{co}: \Gco(M) \rightarrow \G(M)$, $\tau_{sm}: \Gsm(M) \rightarrow \Gco(M)$ and $\tau_{co} \circ \tau_{sm}: \Gsm(M) \rightarrow \G(M)$, defined by $[(u_{\eps})_{\eps}] \mapsto [(u_{\eps})_{\eps}]$ are injective and unital algebra homomorphisms. 
\end{lem}
Whenever convenient, we may therefore omit the natural embeddings and simply write $\Ksm \subseteq \Kco \subseteq \Kt$ and $\Gsm(M) \subseteq \Gco(M) \subseteq \G(M)$.

 Remarkably, $\tau_{co}$ (and therefore also $\tau_{co} \circ \tau_{sm}$) is not surjective, but $\tau_{sm}$ is. Both of these results will be proved below. We start by examining the relation between arbitrary and continuous dependence
on $\eps$. To this end, we first determine the idempotents in the algebra of generalized functions or ring of generalized
numbers, respectively, in the case of continuous or smooth parameter dependence. We first note that the situation
for arbitrary $\eps$-dependence is completely characterized by the following two results: By \cite[Th.\ 4.1]{AJOS08}
the nontrivial idempotents in $\widetilde \K$ are precisely the equivalence classes in $\widetilde \K$ of characteristic
functions $e_S$ of some $S\subseteq I$ with $0\in \overline S \cap \overline{S^c}$. Furthermore, by \cite[Prop.\ 5.3]{V09},
any idempotent of $\G(M)$ for $M$ connected is a generalized constant. 

Contrary to the case of $\G(M)$ and $\widetilde \K$, the following result shows that there are no nontrivial
idempotents in the case of smooth or continuous parameter dependence:

\begin{prop} \label{gen:idem2}
 Let $M$ be a connected smooth manifold. Then there are no nontrivial idempotents in $\Gco(M)$.
\end{prop}

\begin{proof}
 Let $u = [(u_{\eps})_{\eps}] \in \Gco(M)$ such that $u_{\eps} \cdot u_{\eps} = u_{\eps} + n_{\eps}$ for some $(n_{\eps})_{\eps} \in \Nco(M)$.

 We first consider an open, relatively compact and connected open set $U$. There are two possible solutions for the quadratic equation $u_{\eps}(x) \cdot u_{\eps}(x) = u_{\eps} (x) + n_{\eps}(x)$ on U:
\begin{equation} \label{eq:qsol2}
 u_{\eps,1}(x) = \frac{1}{2} + \sqrt{\frac{1}{4} + n_{\eps}(x)} \quad \text{and} \quad u_{\eps,2}(x) = \frac{1}{2} - \sqrt{\frac{1}{4} + n_{\eps}(x)}.
\end{equation}
 As $(n_{\eps})_{\eps}$ is negligible, there exists $\eps_0 >0$ such that $\vert n_{\eps}(x) \vert < \frac{1}{8}$ for all $\eps < \eps_0$ and all $x \in U$.  By continuity of $u$ in $\eps$ and $x$, both of the sets
\begin{align*}
 U_1 := \{ (\eps,x) \in (0,\eps_0] \times U \, | \, u_{\eps}(x) = u_{\eps,1}(x) \} \\
 U_2 := \{ (\eps,x) \in (0,\eps_0] \times U \, | \, u_{\eps}(x) = u_{\eps,2}(x) \}
\end{align*}
 are closed and, as they form a partition of $(0, \eps_0] \times U$, also open in $(0,\eps_0] \times U$. Since the latter is connected we have that either $U_1 = (0,\eps_0] \times U$ or $U_2 = (0,\eps_0] \times U$. Let us assume that it is $U_1$. 
Thus for any $x \in U$, any $m \in \N$ and sufficiently small $\eps$ we obtain that
$$
 \vert u_{\eps}(x) - 1 \vert  
= \left\vert \sqrt{\frac{1}{4} + n_{\eps}(x)} - \frac{1}{2} \right\vert \\
< \eps^m.
$$
 Therefore $\left. u \right\vert_U = 1$ in $\Gco(U)$. In the case $U_2 = (0,\eps_0] \times U$ 
we have that $\left. u \right\vert_U = 0$.

 Now consider
\begin{align*}
 & M_1 := \{ x \in M \, | \, \exists \text{ neighborhood } V \text{ of } x \text{ such that } u \vert_V = 1 \} \\
 & M_2 := \{ x \in M \, | \, \exists \text{ neighborhood } V \text{ of } x \text{ such that } u \vert_V = 0 \}.
\end{align*}
 Both sets are obviously open. Moreover, by the above, $M$ is the disjoint union of $M_1$ and $M_2$. Connectedness of $M$ implies that $u$ is either $1$ or $0$.
\end{proof}

Consequently, there are no nontrivial idempotents in $\Gsm(M)$, $\Kco$ and $\Ksm$.

Next we demonstrate that $\tau_{co}$ is not an isomorphism. Hence $\Kt$ is strictly larger than $\Kco$, and a fortiori $\G(M)$ is strictly larger than $\Gco(M)$. 

\begin{lem} \label{emb2}
 $\tau_{co}: \Kco \rightarrow \Kt$ is not surjective, i.e.\ $\Kco \subsetneq \Kt$.
\end{lem}

\begin{proof}[First proof]
 Let $r = [(r_{\eps})_{\eps}] \in \Kt$ be defined by
\begin{equation*}
 r_{\eps} := \left \lbrace \begin{array}{ll}
				1 & \text{ if } \eps = \frac{1}{n} \text{ for some } n \in \N \\
				0 & \text{else}
                           \end{array} \right. .
\end{equation*}
 Suppose there exists a continuous representative $(s_{\eps})_{\eps}$ of $r$. Then $r_{\eps} = s_{\eps} + n_{\eps}$ for some $(n_{\eps})_{\eps} \in \Neg$. For $\eps$ sufficiently small (say smaller than some $\eps_0 >0$) we have that $\left| n_{\eps} \right| < \frac{1}{4}$ and therefore
\begin{equation} \label{eq:ineq1}
 \text{either} \quad |s_{\eps}| < \frac{1}{4} \quad \text{or} \quad |s_{\eps}| > \frac{3}{4}.
\end{equation}
 For $\N \ni n > \frac{1}{\eps_0}$ we have in particular that
$
 \left| s_{\frac{1}{n}} \right| \geq \left|r_{\frac{1}{n}}\right| - \left|n_{\frac{1}{n}}\right| >  \frac{3}{4}
$
 but (as $\frac{2n+1}{2n(n+1)} = \frac{1}{2}(\frac{1}{n} + \frac{1}{n+1})$)
\begin{equation*}
 \left|s_{\frac{2n+1}{2n(n+1)}}\right| \leq \left|r_{\frac{2n+1}{2n(n+1)}}\right| + \left|n_{\frac{2n+1}{2n(n+1)}}\right| < 0 + \frac{1}{4} = \frac{1}{4}.
\end{equation*}
 By the intermediate value Theorem there must be an $\eps \in ({\frac{2n+1}{2n(n+1)}},\frac{1}{n})$ such that $|s_{\eps}| = \frac{1}{2}$. This contradicts \eqref{eq:ineq1}.
\end{proof}

\begin{proof}[Second proof]
If $\tau_{co}$ was surjective it would be an isomorphism. Since by \cite[Thm.~4.1]{AJOS08} there exist
nontrivial idempotents in $\widetilde\K$, the same would be true of $\Kco$, contradicting Proposition \ref{gen:idem2}.
\end{proof}
This immediately implies:
\begin{cor}
 $\Gco(M) \subsetneq \G(M)$.
\end{cor}


Our next aim is to establish surjectivity of the natural embeddings $\tau_{sm}$, both in the case of 
the rings of generalized numbers $\Kco$ and $\Ksm$ and for the algebras of generalized functions $\Gco$ and $\Gsm$. 



\begin{thm}\label{iso}
 $\Ksm$ is isomorphic to $\Kco$ (via $\tau_{sm}$). 
\end{thm}

\begin{proof}
Let $(r_{\eps})_{\eps} \in \Emco$. By \cite[Lem.~A.9]{MT1997} (or its strengthening, Lemma \ref{approx} below) 
there exists $(s_{\eps})_{\eps} \in \Cinf(I,\K)$ such that
\begin{equation*}
 \vert s_{\eps} - r_{\eps} \vert \leq e^{- \frac{1}{\eps}} \quad \forall \eps \in I,
\end{equation*}
so $\vert s_{\eps} - r_{\eps} \vert < \eps^{m}$ for all $m \in \N$ and $\eps$ sufficiently small.
 This implies $(s_{\eps})_{\eps} \in \Emsm$ and $[(s_\eps)_\eps] = [(r_\eps)_\eps]$ in $\Kco$.
\end{proof}

Alternatively, one could also apply the Weierstra{\ss} Approximation Theorem on compact intervals covering $(0,1]$ to prove Th.\ \ref{iso}.




The proof of surjectivity of $\tau_{sm}: \Gsm(M) \rightarrow \Gco(M)$ will rely on the following extension of \cite[Lem.~A.9]{MT1997}. 

\begin{lem} \label{approx} Let $U\subseteq \R^n$, $W\subseteq \R^m$ be open, and suppose that $h:I\times U \to W$, $(\eps,x) \mapsto h(\eps,x)$ is continuous with respect to $\eps$ and smooth with respect to $x$. Then for any continuous map $g: I\times U \to \R^+$, any $k\in \N_0$ and any open subset $U_1$ of $U$ with $\overline{U_1}\comp U$ there exists a smooth map $f: I\times U \to W$ such that for all $|\alpha|\le k$ and all $\eps\in I$,
$$ \sup_{x\in U_1}\|\partial_x^\alpha h(\eps,x) - \partial_x^\alpha f(\eps,x)\| \le \inf_{x\in U_1} g(\eps,x)$$ 
\end{lem}
\begin{proof} Replacing, if necessary, $g$ by $(\eps,x) \mapsto \min(g(\eps,x), \frac{1}{2}d(h(\eps,x),\R^m\setminus W))$, we may without loss of generality suppose that $W=\R^m$.

 By continuity, for each $\eta\in I$ there exists an open neighborhood $I_\eta$ of $\eta$ in $I$ such that
 $$ \sup_{x\in U_1} \|\partial_x^\alpha h(\eps,x) - \partial_x^\alpha h(\eta,x)\| \le \inf_{x\in U_1} g(\eps,x) \qquad (|\alpha|\le k,\, \eps \in I_\eta)$$
 Choose a smooth partition of unity $(\phi_\eta)_{\eta\in I}$ on $I$ with $\supp \phi_\eta \subseteq I_\eta$ for each $\eta$ and set $f(\eps,x) := \sum_{\eta\in I} \phi_\eta(\eps) h(\eta,x)$. Then $f\in \cinfty(I\times U)$ and for any $\eps\in I$, any $x,\, y\in U_1$ and any $|\alpha|\le k$ we obtain 
 \begin{eqnarray*}
\|\partial_x^\alpha h(\eps,x) - \partial_x^\alpha f(\eps,x)\| &\le& \sum_{\eta\in I} \phi_\eta(\eps)
\|\partial_x^\alpha h(\eps,x) - \partial_x^\alpha h(\eta,x)\|\\
&\le& \sum_{\eta\in I} \phi_\eta(\eps) g(\eps,y) = g(\eps,y),
\end{eqnarray*} 
 so the claim follows.
\end{proof}

\begin{lem} \label{approxem}
 Let $U$, $U_1$ be open subsets of $\R^n$ with $\overline U_1 \comp U$. Then given any $(u_\eps)_\eps \in \Emco(U)$ there exists $(v_{\eps})_\eps \in \Emsm(U_1)$ such that $(u_\eps|_{U_1} - v_{\eps})_\eps \in \Nco(U_1)$.
\end{lem}
\begin{proof}
 By Lemma \ref{approx}, for each $n\in \N_0$ there exists $v_n \in \Cinf(I\times U)$ such that for all $|\alpha|\le n$ and all $\eps\in I$,
\begin{equation}\label{eq:vie}
\sup_{x\in U_1}\|\partial^\alpha u_\eps(x) - \partial^\alpha v_{n,\eps}(x)\| \le e^{-\frac{1}{\eps}}
\end{equation}
 Let $\mathcal{I} =  (I_n)_{n \in \N}$ be the open cover of $I$ defined by $I_n := ( \frac{1}{n+2}, \frac{1}{n} )$ for $n \geq 2$ and $I_1 := (\frac{1}{3}, 1]$. Choose a smooth partition of unity $(\chi_n)_{n \in \N}$ with $\supp{\chi_n} \subseteq I_n$ $\forall n$. For $x \in U_1$ and $\eps \in I$ let
\begin{equation}\label{eq:ve}
 v_{\eps}(x) := \sum_{n=1}^{\infty} \chi_n(\eps) v_{n,\eps}(x).
\end{equation}
 Obviously, $v$ is smooth in $x$ and $\eps$. It remains to be shown that $(v_{\eps} - u_{\eps})_{\eps}$ is negligible on $U_1$. Fix $K\comp U_1$ and $k \in \N_0$. Then for $\eps \leq \frac{1}{k+2}$ and any $\alpha$ with $|\alpha|\le k$ we have that
$$ \sup_{x \in K} \|\partial^\alpha v_{\eps}(x) - \partial^\alpha u_{\eps}(x)\| \stackrel{\eqref{eq:ve}}{\leq}  
  \sum_{n=k+1}^{\infty} \chi_n(\eps) \sup_{x \in U_1}\|\partial^\alpha v_{n,\eps}(x) - 
  \partial^\alpha u_{\eps}(x)\|  \stackrel{\eqref{eq:vie}}{\le} e^{-\frac{1}{\eps}}.
$$
 Thus $(v_{\eps} - u_{\eps})_{\eps} \in \Nco(U_1)$, and therefore also $(v_{\eps})_{\eps} \in \Emsm(U_1)$. 
\end{proof}

 From these preparations we conclude

\begin{thm} \label{GsmGco}
 The map $\tau_{sm}: \Gsm(M) \to \Gco(M)$ is an isomorphism.
\end{thm}
\begin{proof}
 Since both $\Gco(\_)$ and $\Gsm(\_)$ are sheaves of differential algebras we may without loss of generality suppose that $M$ is an open subset of $\R^n$. Furthermore, by Lemma~\ref{gf:emb1}, it remains to be shown that $\tau_{sm}$ is surjective. To this end let $u=[(u_\eps)_\eps]\in \Gco(M)$. Choose a locally finite open cover $(U_\alpha)_{\alpha\in A}$ of $M$ such that $\overline{U_\alpha} \comp M$ for all $\alpha$. Let $(\chi_\alpha)_{\alpha\in A}$ be a partition of unity on $M$ with $\supp \chi_\alpha \subseteq U_\alpha$ for all $\alpha$. By Lemma \ref{approxem}, for each $\alpha\in A$ there exists some $(v_{\alpha,\eps})_\eps\in \Emsm(U_\alpha)$ such that $(u_\eps|_{U_\alpha} - v_{\alpha,\eps})\in \Nco(U_\alpha)$. Then $v_\eps := \sum_\alpha \chi_\alpha v_{\alpha,\eps}$ defines an element $(v_\eps)_\eps$ of $\Emsm(M)$ and by construction, $\tau_{sm}([(v_\eps)_\eps]) = u$.
\end{proof}

The set of generalized numbers $\Ksm$ can be identified with the set of constant generalized functions in $\Gsm(M)$ via $[(r_{\eps})_{\eps}] \mapsto [(u_{\eps})_{\eps}]$, $u_{\eps}(x) := r_{\eps}$ for all $\eps \in I$, $x \in M$. The same is true for $\Kco$ and the set of constant functions in $\Gco(M)$. Thus Theorem \ref{iso} can also be viewed 
as an immediate corollary of Theorem~\ref{GsmGco}.


 A result analogous to Theorem~\ref{GsmGco} also holds for manifold-valued generalized functions from $M$ to $N$. Also in this case we define $\tau_{co}$ and $\tau_{sm}$ to be the natural embeddings, i.e.\ $[(u_{\eps})_{\eps}] \rightarrow [(u_{\eps})_{\eps}]$:
$$\xymatrix{\Gsm[M,N] \;\; \ar@/_1cm/@{^(->}[rrrr]^{\tau_{co} \circ \tau_{sm}} \ar@{^(->}[rr]^{\tau_{sm}} \, && \Gco[M,N] \ar@{^(->}[rr]^{\tau_{co}} \, && \G[M,N]}.$$
 Similarly to Lemma \ref{gf:emb1} we have that these maps are well-defined and injective, using \cite[Def.~2.2 and Def.~2.4]{K2002}. Building on Theorem \ref{GsmGco} we can now show:

\begin{thm}\label{GsmGcomf}
 The map $\tau_{sm}: \Gsm[M,N] \to \Gco[M,N]$ is bijective. 
\end{thm}
\begin{proof}
 By \cite[Prop.~2.2]{KSV2009}, given any Whitney-embedding $i: N\hookrightarrow \R^s$, we may identify $\Gsm[M,N]$ with the subspace $\tilde{\mathcal G}_{sm}[M,i(N)]$ of $\Gsm(M)^s$. 
 The proof of that result carries over verbatim to the $\Gco$-setting. Therefore, we may without loss of generality suppose that $N$ is a submanifold of some $\R^s$. Let $T$ be a tubular neighborhood of $N$ in $\R^s$ with retraction map $r: T \to N$ (cf., e.g., \cite{H1994} or \cite{Lee2003}).

Let $u\in \Gco[M,N]$. By Theorem \ref{GsmGco} there exists $v'\in \Gsm(M)^s$ such that $(u_\eps -  v_\eps')_\eps \in \Nco(M)^s$ (hence, in particular, $v'$ is c-bounded, i.e., $v'\in \Gsm[M,\R^s]$). We will now suitably modify $v'$ such that the resulting element of $\Gsm[M,N]$ equals $u$ in $\Gco[M,N]$. To this end we follow a similar path as that taken in the proof of \cite[Thm.~2.3]{KSV2009}. Let $T'\subseteq T$ be a closed tubular neighborhood of $N$.
 Using a partition of unity subordinate to $\{T,\R^s\setminus T'\}$ we obtain a map $\tilde r: \R^s \to \R^s$ which coincides with $r$ on $T'$. 
 Let $(K_l)_l$ be a compact exhaustion of $M$ with $K_l\subseteq K_{l+1}^\circ$ for all $l$. Since each $u_\eps$ is c-bounded with values in $N$ and $(u_\eps- v_\eps')_\eps$ is negligible, for each $l$ there exists a compact set $K_l' \subseteq T'$ and an $\eps_l>0$ (without loss of generality $\eps_l < \eps_{l-1}$) such that $u_\eps(K_l)\cup  v_\eps'(K_l) \subseteq K_l'$ for all $\eps \le \eps_l$.

 For $\eps\in I$ we set $v_\eps'':= \tilde r \circ v_\eps'$. Then $v'' \in \Gsm[M,\R^s]$ and for $x\in K_l$ and $\eps<\eps_l$ we have (denoting by $\mathrm{ch}(K_l')$ the convex hull of $K_l'$)
\begin{eqnarray*}
\|u_\eps(x) - v_\eps''(x)\| &=& \|\tilde r\circ u_\eps(x) - \tilde r\circ v_\eps'(x)\|\\ 
&\le& \|D \tilde r\|_{L^\infty(\mathrm{ch}(K_l'))} \cdot \|u_\eps(x) - v_\eps'(x)\| = O(\eps^m)
\end{eqnarray*}
 for each $m$ by construction of $v'$. 

 Now let $\eta: M \to \R$ be smooth such that $0<\eta(x)\le \eps_l$ for all $x\in K_l\setminus K_{l-1}^\circ$ ($K_0:=\emptyset$) (cf.\ \cite[Lem.~2.7.3]{GKOS2001}). Moreover, let $\nu: \R^+_0 \to [0,1]$ be a smooth function satisfying 
$\nu(t)\le t$ for all $t$ and
$$
\nu(t) = \left\{
\begin{array}{ll}
t & 0 \le t \le \frac{1}{2} \\
1 & t \ge \frac{3}{2}                     
\end{array}
\right.
$$
 For $(\eps,x)\in I\times M$ let $\mu(\eps,x) := \eta(x) \nu\!\left(\frac{\eps}{\eta(x)}\right)$. Then we may define $v_\eps(x) := v''_{\mu(\eps,x)}(x)$ for $(\eps,x)\in I\times M$. It follows that $v\in \Gsm[M,N]$. Since $v_\eps|_{K_l^\circ} =v_\eps''|_{K_l^\circ}$ for $\eps \le \frac{1}{2} \min_{x\in K_l}\eta(x)$ and any  $l\in \N\,$, $(u_\eps - v_\eps)_\eps$ satisfies the negligibility estimate of order $0$ on any compact subset of $M$. Thus, by \cite[Thm.~3.3]{KSV2003}, we conclude that $u=v$ in $\Gco[M,N]$.
\end{proof}

\begin{rem}
 Similar techniques can be used to show that the smooth and continuous variants of the spaces of generalized vector bundle homomorphisms and of hybrid generalized functions (see \cite{K2002, KS2002, KSV2003} for definitions and characterizations of these spaces) can be identified.
\end{rem}


\section{Algebraic properties of $\Ksm = \Kco$} \label{algprop}

Above we have seen that $\Kco$ and $\Ksm$ are algebraically isomorphic, and are proper subrings of $\Kt$. The aim of the present section is to initiate the investigation of algebraic properties of $\Ksm$ along the lines of
\cite{AJ01,AJOS08,V10}. 
In particular, we point out similarities and differences between the spaces $\Ksm$ and $\Kt$. 


\subsection{Non-invertible elements are zero divisors} \label{zerosec}
By \cite[Sec.\ 2.1]{V10}, $\Kt$ is a reduced ring, i.e., a commutative ring without non-zero nilpotent elements. As
$\Ksm$ is a subring of $\Kt$, it inherits this property

%
%
%

A fundamental property of $\Kt$ is that the non-invertible elements and the zero divisors in $\Kt$ coincide
(see \cite[Thm.\ 2.18]{AJ01}, \cite[Thm.~1.2.39]{GKOS2001}). The same holds true for $\Ksm$:

\begin{prop}
 An element $r \in \Ksm$ is non-invertible if and only if it is a zero divisor.
\end{prop}

\begin{proof}
 
 Let $r$ be non-invertible.
 By \cite[Thm.~1.2.38]{GKOS2001} we have that $r$ is not strictly non-zero (the proof carries
 over unchanged to the $\Ksm$-setting, cf.\ also \cite[Prop.~6.2.5]{B2009} for a generalization), i.e.\ for all representatives $(r_{\eps})_{\eps}$ of $r$ and all $m \in \N$ there exists a strictly decreasing sequence $\eps_k \searrow 0$ such that $\vert r_{\eps_k} \vert < \eps_k^m$. By varying $m$ we obtain a sequence $(\eps_j)_j$, $\eps_j \searrow 0$, such that
\begin{equation*}
 \vert r_{\eps_j} \vert < \eps_j^j \quad \forall j \in \N.
\end{equation*}
 Since $\{ \eps_j \}_{j \in \N}$ is discrete in $(0,1]$ we may find disjoint neighborhoods $(a_j,b_j) \ni \eps_j$ such that
\begin{equation*}
 \vert r_{\eps} \vert < \eps^j \quad \forall \eps \in (a_j,b_j).
\end{equation*}
 On each such interval $(a_j,b_j)$ there exists a smooth bump function $\chi_j \in \mathcal{D}(a_j,b_j)$ such that $\chi_j(\eps_j)=1$, $0\le \chi_j \leq 1$. Let $(s_{\eps})_{\eps}$ be defined by
\begin{equation*}
 s_{\eps} := \left\lbrace \begin{array}{ll}
			\chi_j(\eps) & \text{if } \eps \in (a_j,b_j) \\
			0 & \text{else}
                    \end{array} \right. .
\end{equation*}
 Obviously, $(s_{\eps})_{\eps} \in \Emsm \setminus \Nsm$. Moreover, $(r_{\eps}s_{\eps})_{\eps} \in \Nsm$, hence $r$ is a zero divisor of $\Ksm$.
\end{proof}

\subsection{Exchange rings}

 There are various equivalent definitions for exchange rings (see, e.g., \cite[Sec.~2.2]{V10}). The most convenient one for our purposes is

\begin{defi}
 A commutative ring $R$ with $1$ is an \emph{exchange ring} if for each $r \in R$ there exists an idempotent $e \in R$ such that $r+e$ is invertible.
\end{defi}

 By \cite[Prop.~2.1]{V10}, $\Kt$ is an exchange ring. The situation is different for $\Ksm$.

\begin{lem}
 $\Ksm$ is not an exchange ring.
\end{lem}

\begin{proof}
 By Proposition \ref{gen:idem2}, there are no non-trivial idempotents in $\Ksm$. Moreover, $r\in \Ksm$, defined by $r_{\eps}:=\sin(\frac{1}{\eps})$, is both non-zero and non-invertible, and also $r\pm1$ is non-invertible. 
\end{proof}

\subsection{Gelfand rings}

Both rings, $\Kt$ and $\Ksm$, are Gelfand rings.

\begin{defi}
 A ring $R$ is called \emph{Gelfand ring} if for $a,b \in R$ with $a+b=1$ there exist $r,s \in R$ such that $(1+ar)(1+bs)=0$.
\end{defi}

That $\Kt$ is a Gelfand ring is a direct consequence of the fact that it is an exchange ring. For $\Ksm$
we need a different approach.

\begin{lem}
 $\Ksm$ is a Gelfand ring.
\end{lem}

\begin{proof}
 Assume that $(a_{\eps})_{\eps}$ and $(b_{\eps})_{\eps}$ are representatives of $a$ and $b$ such that $a_{\eps} + b_{\eps} = 1$ for all $\eps$. If $a=0$, then $b=1$, and $r=0$ and $s=-1$ satisfy $(1+ar)(1+bs)=1 \cdot 0 = 0$ (similarly for $b=0$).

 Let $a \neq 0$ and $b \neq 0$. Let $S := \{ \eps \in I : \vert a_{\eps} \vert \geq \frac{1}{2} \}$ and let $\chi \in \Cinf(\R,I)$ such that $\chi \vert_{(-\infty,\frac{1}{2}]} = 0$ and $\chi \vert_{[1,\infty)} = 1$. Then $(r_{\eps})_{\eps}$, defined by
\begin{equation*}
 r_{\eps} := \left\lbrace \begin{array}{ll}
                                - \frac{\chi(2 \vert a_{\eps} \vert)}{a_{\eps}} & \text{if } \eps \in S \\
				0 & \text{else}
                          \end{array} \right. ,
\end{equation*}
 is well-defined and smooth. It is moderate since for $\eps$ such that $\vert a_{\eps} \vert \geq \frac{1}{4}$ we have that 
 $\vert r_{\eps} \vert = \frac{\vert \chi(2 \vert a_{\eps} \vert) \vert}{\vert a_{\eps} \vert} \leq \frac{1}{\vert a_{\eps} \vert} \le 4$, and for $\eps$ such that $\vert a_{\eps} \vert < \frac{1}{4}$ we even have that $\vert r_{\eps} \vert = 0$. Furthermore, $a_{\eps} r_{\eps} = - 1$ on $S$.

 Similarly, there exists $(s_{\eps})_{\eps} \in \Emsm$ such that $b_{\eps} s_{\eps} = - 1$ on 
 $\{ \eps \in I : \vert b_{\eps} \vert \geq \frac{1}{2} \}$.
 Altogether, $(1+a_{\eps}r_{\eps})(1+b_{\eps}s_{\eps})=0$ for all $\eps \in (0,1]$.
\end{proof}

\subsection{Partial order and absolute value}

 The order on $\Rsm$ (and similarly on $\Rco$) is inherited by the order on $\Rt$ \cite[Sec.~1.2.4]{GKOS2001}:

\begin{defi} \label{po}
 Let $r,s \in \Rsm$. We write $r \leq s$ if there are representatives $(r_{\eps})_{\eps}$, $(s_{\eps})_{\eps}$ with $r_{\eps} \leq s_{\eps}$ for all $\eps$.
\end{defi}

\begin{rem} \label{order2}
 Note that this is equivalent to the fact that for any representatives $\bar r$, $\bar s$ of $r$ and $s$ there exists some $(n_{\eps})_{\eps} \in \Nsm$ with $\bar{r}_{\eps} \leq \bar{s}_{\eps} + n_{\eps}$. 

Moreover, by \cite{OPS07}, $r \le s$ if and only if for all representatives $(r_\eps)$, $(s_\eps)_\eps$ and any $a>0$
there exists some $\eps_0>0$ such that $r_\eps \le s_\eps + \eps^a$ for all $\eps<\eps_0$. Further properties
of the order structure in $\Rt$ and $\G$ can be found in \cite{OPS07,M08}.

\end{rem}
The same argument as in the case of $\Rt$ yields: 
\begin{prop}
 $(\Rsm, \leq)$ is a partially ordered ring.
\end{prop}

 By the identification of $\Ksm$ with $\Kco$ in Theorem~\ref{iso} we can even define the absolute value of generalized numbers in $\Ksm$ (note that generally $(\vert r_{\eps} \vert)_{\eps} \in \Emco$ but $\not\in\Emsm$):

\begin{defi} \label{absval}
 Let $r = [(r_{\eps})_{\eps}] \in \Ksm$. The \emph{absolute value} of $r$, denoted by $|r|_{sm}$, is defined as the generalized number
\begin{equation*}
 \vert r \vert := \tau_{sm}^{-1}([(|r_{\eps}|)_{\eps}]),
\end{equation*}
 where $| \_ |$ denotes the absolute value in $\C$ and $\tau_{sm}: \Ksm \rightarrow \Kco$ is the canonical isomorphism (cf.\ Theorem \ref{iso}).
\end{defi}
 
 By the identification of $\Rsm$ with $\Rco$ we can show that it is a lattice.

\begin{defi}
 A \emph{lattice} is a partially ordered set $R$ such that any two elements $r,s \in R$ have a \emph{join} (or \emph{supremum}) $r \vee s$ and a \emph{meet} (or \emph{infimum}) $r \wedge s$.

 A partially ordered ring that is a lattice for this order is called an \emph{$l$-ring} (or \emph{lattice ordered ring}).
\end{defi}

\begin{defi} \label{minmax}
 The \emph{minimum} $\min (r,s)$ and the \emph{maximum} $\max (r,s)$ for $r=[(r_{\eps})_{\eps}],s=[(s_{\eps})_{\eps}] \in \Rsm$ are defined as follows:
\begin{eqnarray*}
 & & \min(r,s) := \tau_{sm}^{-1}([(\min(r_{\eps},s_{\eps}))_{\eps}]) \\
 & & \max(r,s) := \tau_{sm}^{-1}([(\max(r_{\eps},s_{\eps}))_{\eps}])
\end{eqnarray*}
\end{defi}

 These notions are well-defined for $\Rsm$ since the $\min$ and $\max$ of real-valued continuous functions are continuous themselves. Clearly, $\min((r_{\eps})_{\eps}+\Nco,(s_{\eps})_{\eps}+\Nco) = (\min(r_{\eps},s_{\eps}))_{\eps}+\Nco$ etc. Thus by Remark \ref{order2} we have:

\begin{lem}
 The minimum and maximum as defined above are well-defined and compatible with the partial order structure of $(\Rsm, \leq)$.
\qed
\end{lem}

 This result is remarkable since the underlying ring in the definition of $\Rsm$---namely $\Cinf(I,\R)$---does not satisfy these properties. Setting $r\vee s = \max(r,s)$ and $r\wedge s = \min(r,s)$ we obtain (cf.\ \cite[Sec.\ 2.3]{V10}):

\begin{prop}
$\Rsm$ is an $l$-ring.
\end{prop}
Clearly, the absolute value as introduced in Definition \ref{absval} is compatible with the order structure on $\Rsm$,
i.e., $|r| = \max(r,-r)$ for any $r\in \Rsm$.

\begin{defi}
 A commutative ring $R$ with $1$ is called \emph{$f$-ring} if it is an $l$-ring and for all $r,s,t \in R$ with $t \geq 0$: $(r \wedge s) t = rt \wedge st$.
\end{defi}

 By \cite[Prop.~2.2]{V10}, $\Rt$ is an $f$-ring. The same holds true for $\Rsm$:

\begin{prop} \label{fring}
$\Rsm$ is an $f$-ring.
\end{prop}

\begin{proof}
 Let $r,s,t \in \Rco$ with representatives $(r_{\eps})_{\eps}$ of $r$, $(s_{\eps})_{\eps}$ of $s$ and $(t_{\eps})_{\eps}$ of $t$ such that $t_{\eps} \geq 0$ for all $\eps$. Then $\min(r_{\eps},s_{\eps})t_{\eps} = \min(r_{\eps}t_{\eps},s_{\eps}t_{\eps})$ for all $\eps \in I$, and therefore $(r \wedge s)t = rt \wedge st$. By Theorem \ref{iso} and Definition \ref{minmax} the claim
follows.
\end{proof}

 For some properties of $l$- and $f$-rings see \cite{Ban2004}.

A main technical tool in the algebraic investigation of $\Kt$ are characteristic functions $e_S$ of subsets
$S\subseteq I$ (cf.\ \cite{AJ01,AJOS08,V09,V10}). Obviously, such functions on $I$ are not continuous unless
in trivial cases and therefore cannot be utilized in the $\Ksm$-setting. This forecloses a direct adaptation
of many techniques of proof from $\Kt$ to $\Ksm$. In certain situations, however, a substitute
for these techniques can be based on the notion of characteristic set. 
\begin{defi}
 A subset $S$ of $I$ is called \emph{characteristic set} if $0 \in \overline{S}$.
\end{defi}
If $r\in \Ksm$ and $S$ is a characteristic set then by $r|_S= 0$ we mean that for any $m\in \N$
there exists some $\eps_0$ such that $|r_\eps|<\eps^m$ for all $\eps\in S$ with $\eps<\eps_0$
(which clearly is independent of the representative of $r$).

\begin{lem} \label{charS}
 Let $r,s \in \Ksm$, $r,s \neq 0$ and $rs = 0$. Then there exists a characteristic set $S$ such that $\left. r \right\vert_S = \left. s \right\vert_S = 0$.
\end{lem}

\begin{proof}
 Since $r$ and $s$ are non-zero there exist characteristic sets $S_r$ and $S_s$ and $K \in \N$ such that
\begin{equation} \label{eq:eps1}
 \vert r_{\eps} \vert > \eps^K \; \forall \eps \in S_r \quad \text{and} \quad \vert s_{\eps} \vert > \eps^K \; \forall \eps \in S_s.
\end{equation}
 Let $m = 2K$. Since $rs=0$ there exists $\eps_m>0$ such that
\begin{equation} \label{eq:eps2}
 \vert r_{\eps} s_{\eps} \vert < \eps^m \quad \forall \eps < \eps_m.
\end{equation}
 Moreover, $S_r$ and $S_s$ are disjoint on $(0,\eps_m)$, i.e.\ $S_r \cap S_s \cap (0,\eps_m) = \emptyset$: For $\eps \in S_r$, $\eps < \eps_m$ we have that $\eps^K \vert s_{\eps} \vert < \vert r_{\eps} s_{\eps} \vert < \eps^m$. Therefore $\vert s_{\eps} \vert < \eps^{m-K} = \eps^K$, i.e.\ $\eps \notin S_s$ by \eqref{eq:eps1}.

 For all $\eps \in S_r \cap (0,\eps_m)$ we have by \eqref{eq:eps1} and \eqref{eq:eps2} that $\vert r_{\eps} \vert - \vert s_{\eps} \vert > \eps^K - \eps^K = 0$. In particular, since $S_r \cap (0,\eps_m) \neq \emptyset$ ($S_r$ being a characteristic set), there exists $\eps_r < \eps_m$ such that $\vert r_{\eps_r} \vert - \vert s_{\eps_r}\vert > 0$. Similarly, there exists $\eps_s < \eps_m$ such that $\vert r_{\eps_s} \vert - \vert s_{\eps_s} \vert < 0$. Hence by continuity in $\eps$ there exists $\delta_m \in (0,\eps_m)$ such that $\vert r_{\delta_m} \vert = \vert s_{\delta_m} \vert$. We even know that $\delta_m \notin S_r \cup S_s$. In fact, as $\delta_m < \eps_m$ equation \eqref{eq:eps2} implies that $\vert r_{\delta_m} s_{\delta_m} \vert < \delta_m^m$, and therefore $\vert r_{\delta_m} \vert = \vert s_{\delta_m} \vert < \delta_m^{\frac{m}{2}}$.

 To construct the characteristic set $S$ we proceed by induction. Let $\overline{\eps}_1 := \delta_m = \delta_{2K}$. Suppose we have already constructed $\overline{\eps}_i < \min(\overline{\eps}_{i-1}, \frac{1}{i})$ such that
\begin{equation} \label{eq:eps3}
 \vert r_{\overline{\eps}_i} \vert = \vert s_{\overline{\eps}_i} \vert < \overline{\eps}_i^{\frac{m+2(i-1)}{2}}.
\end{equation}
 As above we find  $\eps_{m+2i} < \min(\overline{\eps}_i, \frac{1}{i+1})$ such that $\vert r_{\eps} s_{\eps} \vert < \eps^{m+2i}$ for all $\eps < \eps_{m+2i}$. 
Since  $m+2i > 2K$ all other arguments hold as well and we finally obtain $\overline{\eps}_{i+1} < \min(\overline{\eps}_i, \frac{1}{i+1})$ such that \eqref{eq:eps3} holds for $i+1$ instead of $i$.

 Since $\overline{\eps_j} \searrow 0$ we have that $S := \{ \overline{\eps}_j \; | \; j \in \N \}$ is a characteristic set and by \eqref{eq:eps3} it follows that $\left. r \right\vert_S = \left. s \right\vert_S = 0$.
\end{proof}

Let $S \subseteq I$ be a characteristic set and let $\A$ denote the algebra $\Ksm$ or $\Gsm(M)$. An element $u \in \A$ is called invertible with respect to $S$ if there exists $v \in \A$ and $r \in \Ksm$ such that $$uv = r1 \text{ in } \A \text{ and } 
\left. (r - 1)\right\vert_S = 0 \text{ in } \Ksm.$$

By \cite[Prop.\ 4.2]{B2010}, an element $r$ of $\Ksm$ is nonzero if and only if there exists a characteristic set
$S$ such that $r$ is invertible with respect to $S$. The proof of this result also shows that $r$ 
is invertible with respect to $S$
if and only if it is strictly nonzero with respect to $S$, which gives a generalization of \cite[Prop.\ 1.2.38]{GKOS2001}.
Analogous results hold for generalized functions (\cite[Sec.\ 6.2]{B2009}).

\begin{defi}
 Let $R$ be a ring and $r \in R$. The \emph{annihilator} of $r$ is defined as the set $\Ann{r} := \{ s \in R : rs=0 \}$.
\end{defi}

\begin{thm} \label{normal}
 Let $r,s \in \Ksm$. The following are equivalent:
\begin{enumerate}
 \item $rs = 0$
 \item there exists $x \in \Ksm$ such that $rx = 0$ and $s(1-x)=0$
 \item $\Ann{r} + \Ann{s} = \Ksm$.
 \item $|r| \wedge |s| = 0$.
\end{enumerate}
\end{thm}

 We show this along the lines of the proof of \cite[Lem.~2.3]{V10}, where the result was verified for $\Kt$.

\begin{proof}
 $(i) \Rightarrow (ii)$ Let $rs = 0$. 
The cases $r=0$ or $s=0$ being obvious, we may assume that both $r$ and $s$ are zero divisors. For all $m \in \N$ there exists $\eps_m > 0$ such that $\vert r_{\eps} s_{\eps} \vert < \eps^m$ for all $\eps < \eps_m$ by $(i)$. Without loss of generality we can assume that $(\eps_m)_m$ is a decreasing sequence. Moreover, by moderateness of $(r_{\eps})_{\eps}$ and $(s_{\eps})_{\eps}$ we have an $N \in \N$ such that $\vert r_{\eps} \vert < \eps^{-N}$ and $\vert s_{\eps} \vert < \eps^{-N}$ for $\eps$ sufficiently small. 
Using a partition of unity argument (cf., e.g., \cite[Lem.\ 2.7.3]{GKOS2001}) we obtain a function $\eta \in \Cinf(I,\R)$ such that
\begin{equation*}
 0 < \eta(\eps) \leq \eps^{m+N} \quad \text{for } \eps \in [\eps_{m+1},\eps_m].
\end{equation*}
Let
\begin{eqnarray*}
 U & := & \{ \eps \in I : \vert r_{\eps} \vert < \vert s_{\eps} \vert + \eta(\eps) \} \\
 V & := & \{ \eps \in I : \vert r_{\eps} \vert \leq \vert s_{\eps} \vert - \eta(\eps) \}
\end{eqnarray*}
By continuous dependence of $r,s$ and $\eta$ on $\eps$, $U$ is open and $V$ is closed in $I$. Using a partition of unity
subordinate to $\{I\setminus V,\, U\}$ we obtain a smooth bump function $I\to [0,1]$, $\eps\mapsto x_\eps$ with
$\left. x \right\vert_V = 1$, $\left. x \right\vert_U \leq 1$ and $\left. x \right\vert_{I \setminus U} = 0$.
In particular, $(x_{\eps})_{\eps} \in \Emsm$.

 Therefore we have that (using $\Ksm \cong \Kco$ by Theorem \ref{iso} and calculating in $\Emco$)
\begin{eqnarray*}
 0 \;\; \leq \;\; (\vert r \vert x)^2 & = & \left\lbrace \begin{array}{ll}
					 (\vert r_{\eps} \vert x_{\eps})^2 & \text{if }\eps \in U \\
					 0 & \text{else}
                                        \end{array} \right. \\
 & \stackrel{x_{\eps}^2 \leq 1}{\leq} & \left\lbrace \begin{array}{ll}
		    \vert r_{\eps} \vert (\vert s_{\eps} \vert + \eps^{m+N}) & \text{if } \eps \in U \\
		    0 & \text{else}
                   \end{array} \right. \\
 & < & 2 \eps^m \quad \text{for } \eps \text{ sufficiently small}
\end{eqnarray*}
since $\vert r_{\eps} s_{\eps} \vert < \eps^m$ and $\vert r_{\eps} \vert \eps^{m+N} < \eps^{-N} \eps^{m+N} = \eps^m$
for such $\eps$. Hence
\begin{equation*}
 r x = 0.
\end{equation*}
 Similarly,
\begin{eqnarray*}
 0 \;\; \leq \;\; (\vert s \vert (1-x))^2 & = & \left\lbrace \begin{array}{ll}
					 (\vert s_{\eps} \vert (1- x_{\eps}))^2 & \text{if }\eps \notin V \\
					 0 & \text{else}
                                        \end{array} \right. \\
 & \stackrel{(1-x_{\eps})^2 \leq 1}{\leq} & \left\lbrace \begin{array}{ll}
		    \vert s_{\eps} \vert (\vert r_{\eps} \vert + \eps^{m+N}) & \text{if } \eps \notin V \\
		    0 & \text{else}
                   \end{array} \right. \\
 & < & 2 \eps^m \quad \text{for } \eps \text{ sufficiently small}.
\end{eqnarray*}
 Thus also
\begin{equation*}
 s (1-x) = 0.
\end{equation*}

 $(ii) \Rightarrow (iii)$ By $(ii)$ there exists $x \in \Ksm$ such that $x \in \Ann{r}$ and $1-x \in \Ann{s}$. For any
$t \in \Ksm$, $t = xt + (1-x)t$. Since annihilators are ideals in the ring, $xt \in \Ann{r}$ and $(1-x)t \in \Ann{s}$. 

 $(iii) \Rightarrow (i)$ By $(iii)$ we may write $1 = x + y$ for $x \in \Ann{r}$ and $y \in \Ann{s}$. Therefore
\begin{equation*}
 rs = rs 1 = rs (x+y) = (rx)s + r(sy) = 0.
\end{equation*}

 $(i) \Leftrightarrow (iv)$ As $rs=0$ is equivalent to $\vert r \vert \vert s \vert = 0$, we may assume that $r,s \in \Rsm$. By Sec.\ \ref{zerosec} $\Ksm$ is a reduced ring and by Proposition \ref{fring} it is an $f$-ring. Since the equivalence holds in any reduced f-ring (see \cite[Thm.~9.3.1]{BKW77}) we are done.
\end{proof}

 From the equivalence of $(i)$ and $(iii)$ we can deduce another property of rings of generalized numbers, namely normality. Since we are dealing with reduced rings, we may use the following definition (cf.\ \cite[Sec.\ 2.3]{V10} for different equivalent conditions):

\begin{defi}
 A reduced commutative $f$-ring $R$ with $1$ is called \emph{normal} if for all $r,s \in R$ with $rs=0$ we can write $R = \Ann{r} + \Ann{s}$.
\end{defi}

\begin{cor}
 $\Rt$ and $\Rsm$ are (reduced) normal $f$-rings.
\end{cor}

\begin{proof}
The property of being a reduced ring was noted at the beginning of Sec.\ \ref{zerosec}, the other claims follow from Proposition~\ref{fring} and Theorem~\ref{normal}.
\end{proof}

\subsection{Ideals}

 In recent years, various properties of ideals in the ring $\Kt$ of generalized numbers have been studied. Previous investigations have led, among others, to a complete description of the maximal ideals (see \cite[Thm.~4.20]{AJ01}), minimal prime ideals (see \cite[Cor.~4.7]{AJOS08}) and prime ideals (see \cite[Thm.~3.6]{V10}) in $\Kt$. In this section we initiate a similar study for the ring $\Ksm$ of generalized numbers with smooth parameter dependence and provide some basic properties of its ideals.

  Let $R$ be a commutative ring with $1$. An ideal $J$ in $R$ is denoted by $J \trianglelefteq R$, a proper ideal is denoted by $J \vartriangleleft R$. 
 Moreover, we call $J \vartriangleleft R$ \emph{prime} if for all $r,s \in R$ with $rs \in J$ we have that $r \in J$ or $s \in J$. 
A proper ideal $J$ is called \emph{maximal} if the only ideal properly containing it is $R$ itself.
$J \trianglelefteq R$ is called \emph{idempotent} if $J^2 = J$.

 The radical of an ideal $J \vartriangleleft R$ is denoted by $\sqrt{J} = \{r \in R | \exists n \in N : x^n \in J \}  = \bigcap_{\substack{J \subseteq P \\ P~\text{prime}}} P$ (see, for example, \cite[Cor.~0.18]{GJ60}). An ideal $J \trianglelefteq R$ is called \emph{radical} if $J = \sqrt{J}$.
To begin with, we investigate convexity of ideals in $\Ksm$.
\begin{defi}
 Let $R$ be a partially ordered ring and $J \trianglelefteq R$ an ideal. $J$ is said to be \emph{convex} if $0 \leq y \leq x$ and $x \in J$ imply that $y \in J$.

 An ideal $J$ in an l-ring $R$ is called \emph{absolutely convex} (or \emph{l-ideal}) if $|y| \leq |x|$ and $x \in J$ imply $y \in J$.
\end{defi}

 In \cite[Prop.~3.7]{AJOS08} it was shown that every ideal in $\Kt$ is absolutely convex. For $\Rsm$ we firstly have

\begin{prop} \label{convex}
 All ideals in $\Rsm$ are convex.
\end{prop}
\begin{proof}
 Let $J \trianglelefteq \Rsm$, $x \in J$ and $0 \leq y \leq x$. Without loss of generality we may consider representatives $(x_{\eps})_{\eps}$, $(y_{\eps})_{\eps}$ such that $0 < y_{\eps} \leq x_{\eps}$ for all $\eps \in I$ (otherwise add $(e^{-\frac{1}{\eps}})_{\eps} \in \Nsm$ to non-negative representatives). Thus $(a_{\eps})_{\eps}$ defined by
$$a_{\eps} := \frac{y_{\eps}}{x_{\eps}} \qquad \forall \eps \in I$$
 is well-defined, smooth and bounded by $1$, hence moderate. Since $x \in J$ and  $y = a x$ we also have that $y \in J$. 
\end{proof}

 In order to prove that ideals are in fact absolutely convex, we show the following Lemma on $\Rsm$ and $\Csm$:

\begin{lem} \label{absc}
 Let $J \unlhd \Ksm$ and $x \in J$. Then $|x| \in J$.
\end{lem}
\begin{proof}
 According to Theorem~\ref{iso} we can work in $\Cco$. The proof for $\Rco$ proceeds along the same lines. Let $(x_{\eps})_{\eps} \in \Emco$ be a representative of $x$. We construct $(a_{\eps})_{\eps} \in \Emco$ such that $a x = |x|$.

 Fix $m \in \N$. Let $b_m: (0,1] \rightarrow (0,1]$ be defined by
\begin{eqnarray} \label{eq:b}
 b_m(\eps) & := & \left\lbrace \begin{array}{ll}
                                \frac{\eps^m}{|x_{\eps}|} & \text{if } |x_{\eps}| \geq \eps^m \\
                                1 & \text{else}
                               \end{array}
\right. .
\end{eqnarray}
Then $b_m$ is continuous. 
In order to obtain the necessary asymptotic behavior, we patch the $b_m$ together. Thus we consider the open cover $\mathcal{I} := \{ (\frac{1}{m+1},\frac{1}{m-1}) \}_{m > 1} \cup \{ (\frac{1}{3},1] \}$ of the interval $(0,1]$, and a corresponding (continuous) partition of unity $( \chi_m )_{m \in \N}$. By $\arg(z)$ we denote the argument of the complex number $z$. Let
\begin{eqnarray*}
 a_{\eps} & := & \left\lbrace \begin{array}{ll}
                               e^{- i \arg(x_{\eps})}( 1 - \sum_{m=1}^{\infty} b_m(\eps) \chi_m(\eps)) & \text{if } x_{\eps} \neq 0 \\
			       0 & \text{if } x_{\eps} = 0
                              \end{array}
 \right. \qquad \forall \eps \in (0,1].
\end{eqnarray*}
Suppose that $x_{\bar\eps} = 0$, $x_{\eps_k} \neq 0$ and $\eps_k\to \bar\eps$. 
Then
\begin{equation*}
 \lim_{k\to \infty} a_{\eps_k} = \lim_{k\to \infty} \big( \underbrace{e^{- i \arg(x_{\eps_k})}}_{|.| \leq 1} (1 - \sum_{m \in \N} b_m(\eps_k) \chi_m(\eps_k)) \big) = 0
\end{equation*}
due to \eqref{eq:b}. Thus $(a_{\eps})_{\eps} \in \co(I,\C)$. Furthermore, $(a_{\eps})_{\eps}$ is moderate:
\begin{equation} \label{eq:a}
 |a_{\eps}|  \leq  |e^{- i \arg(x_{\eps})}| \cdot |1-\sum_{m=0}^{\infty} b_m(\eps) \chi_m(\eps)| \le 2\\
\end{equation}
 It remains to show that $(a_{\eps} x_{\eps} - |x_{\eps}|)_{\eps} \in \Nco$. Since all terms are continuous in $\eps$ and 
$x_{\eps} = e^{i \arg(x_{\eps})} |x_{\eps}|$ it is sufficient to consider
\begin{equation} \label{eq:ax}
 |a_{\eps} - e^{-i \arg(x_{\eps})}||x_{\eps}|
\end{equation}
 in the following cases (we assume that $\eps \in (\frac{1}{m+1},\frac{1}{m}]$ throughout):
\begin{itemize}
 \item $|x_{\eps}| < \eps^{m+1}$: By \eqref{eq:b} and \eqref{eq:a}, $a_{\eps} 
 = 0$, so \eqref{eq:ax}$\, = 1 \cdot |x_{\eps}| < \eps^{m+1}$.
 \item $\eps^{m+1} \leq |x_{\eps}| < \eps^{m}$: In this case $a_{\eps} = e^{-i \arg(x_{\eps})} (1 - \frac{\eps^{m+1}}{|x_{\eps}|}\chi_{m+1}(\eps) - \chi_m(\eps))$, so \eqref{eq:ax}$ \, \le \big( \frac{\eps^{m+1}}{|x_{\eps}|} + 1 \big) |x_{\eps}| < \eps^{m+1} + \eps^m$.
 \item $|x_{\eps}| \geq \eps^m$: Here, $a_{\eps} = e^{-i \arg(x_{\eps})} \big(1 - \frac{\eps^{m+1}}{|x_{\eps}|}\chi_{m+1}(\eps) - \frac{\eps^{m}}{|x_{\eps}|}\chi_{m}(\eps) \big)$, so as above \eqref{eq:ax}$\, \leq \eps^{m+1} + \eps^m$.
\end{itemize}
 Summing up, we obtain for all $m \in \N$ that
\begin{equation}
 |a_{\eps} - e^{-i \arg(x_{\eps})}| |x_{\eps}| < 2 \eps^m \qquad \text{ for } \eps \leq \frac{1}{m}.
\end{equation}
 Thus $(a_{\eps}x_{\eps} - |x_{\eps}|)_{\eps} \in \Nco$, and hence $|x| = [(|x_{\eps}|)_{\eps}] \in J$.
\end{proof}

\begin{prop} \label{absc1}
 All ideals in $\Rsm$ are absolutely convex.
\end{prop}
\begin{proof}
 By Proposition~\ref{convex}, all ideals in $\Rsm$ are convex. According to \cite[Thm.~5.3]{GJ60}, a convex ideal $J \unlhd \Rsm$ is absolutely convex if and only if $x \in J$ implies that $|x| \in J$. This is Lemma~\ref{absc}.
\end{proof}

 Moreover, we can deduce from Lemma~\ref{absc} that all finitely generated ideals in $\Rsm$ and $\Csm$ are in fact principal ideals. 
\begin{prop}
 Let $r, s \in \Ksm$. Then
\begin{enumerate}
 \item $r \Ksm + s \Ksm = (|r|+|s|)\Ksm = (|r|\vee|s|)\Ksm$
 \item $r \Ksm \cap s \Ksm = (|r|\wedge|s|)\Ksm$.
\end{enumerate}
\end{prop}

\begin{proof}
 Both statements can be proved along the same lines as the corresponding ones for 
$\Kt$ in \cite[Lem.~3.1]{V10}: $\Rsm$ is an f-ring (by Proposition~\ref{fring}), and all ideals are absolutely convex (by Proposition~\ref{absc1}). Thus $(i)$ and $(ii)$ follow from \cite[Prop.~8.2.8]{BKW77} and \cite[Prop.~9.1.8]{BKW77}, respectively. The results can be transferred to ideals in $\Csm$ by using the bijective correspondence between ideals in $\Csm$ and $\Rsm$ (analogous to \cite[Sec.~2.4]{V10}).
\end{proof}

 Furthermore, we can characterize powers and radicals of ideals in $\Ksm$. In what follows, 
 $\langle A \rangle$ denotes the ideal generated by $A$. 
 
\begin{lem}
 Let $J \trianglelefteq \Ksm$ and $m \in \N$. Then the following properties hold:
\begin{enumerate}
 \item $J^m = \{ r \in \Ksm : \sqrt[m]{|r|} \in J \}$.
 \item Let $L \trianglelefteq \Ksm$ and $L^m \subseteq J^m$. Then $L \subseteq J$. In particular, if $r \in \Ksm$ and $r^m \in J^m$, then $r \in J$.
 \item $\sqrt{J} = \langle \sqrt[n]{|r|} : n \in \N, r \in J \rangle$, and in particular, for $s \in \Ksm$, $\sqrt{s \Ksm} = \langle \sqrt[n]{|s|} : n \in \N \rangle$.
\end{enumerate}
\end{lem}

\begin{proof}
 Extracting roots is a continuous function, and hence is an inner operation in $\Kco$ and therefore $\Ksm$. Thus the proof is identical to the case of arbitrary parametrization by making use of \cite[Prop.~8.2.11]{BKW77}. See \cite[Lem.~3.2]{V10} for details.
\end{proof}
 The idempotent ideals are exactly the radical ideals:

\begin{prop} \label{idra}
 Let $J \trianglelefteq \Ksm$. The following are equivalent:
\begin{enumerate}
 \item $J$ is idempotent.
 \item $J$ is radical.
 \item $\forall r \in J: \sqrt{|r|} \in J$.
 \item $J$ is an intersection of prime ideals.
\end{enumerate}
\end{prop}

\begin{proof}
 Identical to that of \cite[Prop.~3.3.]{V10}.
\end{proof}

 The next result shows, in particular, that the sum and the intersection of a family of radical ideals is again radical (see $(i)$ and $(iv)$).

\begin{prop}
 Let $J_{\lambda} \trianglelefteq \Ksm$ for all $\lambda\in \Lambda$. Then
\begin{enumerate}
 \item $\sqrt{\sum_{\lambda \in \Lambda} J_{\lambda}} = \sum_{\lambda \in \Lambda} \sqrt{J_{\lambda}}$
 \item Let $I, J \trianglelefteq \Ksm$. Then $\sqrt{I} \cap \sqrt{J} = \sqrt{I \cap J}$. 
 \item Let $J \trianglelefteq \Ksm$. Then
$$J^{\sqrt{}} := \bigcap_{n \in \N} J^n = \{ r \in \Ksm \, | \, \forall n \in \N: \sqrt[n]{|r|} \in J \} = \{r \in \Ksm \, | \, \sqrt{r \Ksm} \subseteq J \}$$
 is the largest radical ideal that is contained in $J$. $J$ is radical if and only if $J = J^{\sqrt{}}$. 
 \item $\bigcap_{\lambda \in \Lambda} J_{\lambda}^{\sqrt{}} = \left( \bigcap_{\lambda \in \Lambda} J_{\lambda} \right)^{\sqrt{}}$
\end{enumerate}
\end{prop}

\begin{proof} Based on the above results, this is analogous to \cite[Prop.~3.4]{V10}. 
\end{proof}

\begin{rem}
 We have seen that many characterizations of ideals in $\Ksm$ can be carried over from $\Kt$. The characterization of prime ideals, however, heavily relies on the structure of $\Kt$ and makes use of the idempotents therein, see \cite[Thm.~3.5, Thm.~3.6]{V10}. Thus a characterization of prime ideals in $\Ksm$ will have to go along different lines.
\end{rem}





\bibliographystyle{amsplain}

\providecommand{\bysame}{\leavevmode\hbox to3em{\hrulefill}\thinspace}
\providecommand{\MR}{\relax\ifhmode\unskip\space\fi MR }
\providecommand{\MRhref}[2]{%
  \href{http://www.ams.org/mathscinet-getitem?mr=#1}{#2}
}
\providecommand{\href}[2]{#2}


\end{document}